\documentclass[12pt]{amsart}
\usepackage{amsmath,amssymb,amsbsy,amsfonts,amsthm,latexsym,
                        amsopn,amstext,amsxtra,euscript,amscd,mathrsfs,color}

\usepackage{float}
\usepackage[english]{babel}
\usepackage{url}
\usepackage[colorlinks,linkcolor=blue,anchorcolor=blue,citecolor=blue]{hyperref}

\begin{document}

\newtheorem{theorem}{Theorem}
\newtheorem{lemma}[theorem]{Lemma}
\newtheorem{example}[theorem]{Example}
\newtheorem{algol}{Algorithm}
\newtheorem{corollary}[theorem]{Corollary}
\newtheorem{prop}[theorem]{Proposition}
\newtheorem{defin}[theorem]{Definition}
\newtheorem{question}[theorem]{Question}
\newtheorem{problem}[theorem]{Problem}
\newtheorem{remark}[theorem]{Remark}
\newtheorem{conjecture}[theorem]{Conjecture}

\newcommand{\comm}[1]{\marginpar{%
\vskip-\baselineskip 
\raggedright\footnotesize
\itshape\hrule\smallskip#1\par\smallskip\hrule}}


\def\cA{{\mathcal A}}
\def\cB{{\mathcal B}}
\def\cC{{\mathcal C}}
\def\cD{{\mathcal D}}
\def\cE{{\mathcal E}}
\def\cF{{\mathcal F}}
\def\cG{{\mathcal G}}
\def\cH{{\mathcal H}}
\def\cI{{\mathcal I}}
\def\cJ{{\mathcal J}}
\def\cK{{\mathcal K}}
\def\cL{{\mathcal L}}
\def\cM{{\mathcal M}}
\def\cN{{\mathcal N}}
\def\cO{{\mathcal O}}
\def\cP{{\mathcal P}}
\def\cQ{{\mathcal Q}}
\def\cR{{\mathcal R}}
\def\cS{{\mathcal S}}
\def\cT{{\mathcal T}}
\def\cU{{\mathcal U}}
\def\cV{{\mathcal V}}
\def\cW{{\mathcal W}}
\def\cX{{\mathcal X}}
\def\cY{{\mathcal Y}}
\def\cZ{{\mathcal Z}}

\def\C{\mathbb{C}}
\def\F{\mathbb{F}}
\def\K{\mathbb{K}}
\def\Z{\mathbb{Z}}
\def\R{\mathbb{R}}
\def\Q{\mathbb{Q}}
\def\N{\mathbb{N}}
\def\M{\textsf{M}}
\def\PP{\mathbb{P}}
\def\A{\mathbb{A}}
\def\p{\mathfrak{p}}
\def\n{\mathfrak{n}}
\def\X{\mathcal{X}}
\def\x{\textrm{\bf x}}
\def\w{\textrm{\bf w}}

\def\({\left(}
\def\){\right)}
\def\[{\left[}
\def\]{\right]}
\def\<{\langle}
\def\>{\rangle}

\def\gen#1{{\left\langle#1\right\rangle}}
\def\genp#1{{\left\langle#1\right\rangle}_p}
\def\genPs{{\left\langle P_1, \ldots, P_s\right\rangle}}
\def\genPsp{{\left\langle P_1, \ldots, P_s\right\rangle}_p}

\def\e{e}

\def\eq{\e_q}
\def\fh{{\mathfrak h}}

\def\lcm{{\mathrm{lcm}}\,}

\def\fl#1{\left\lfloor#1\right\rfloor}
\def\rf#1{\left\lceil#1\right\rceil}
\def\mand{\qquad\mbox{and}\qquad}

\def\jt{\tilde\jmath}
\def\ellmax{\ell_{\rm max}}
\def\llog{\log\log}

\def\ch{\hat{h}}
\def\GL{{\rm GL}}
\def\Orb{\mathrm{Orb}}
\def \S{\mathcal{S}}
\def\vec#1{\mathbf{#1}}
\def\ov#1{{\overline{#1}}}
\def\Gal{{\rm Gal}}

\numberwithin{equation}{section}
\numberwithin{theorem}{section}

\title[Quantitative dynamical Mordell-Lang conjecture]{On the quantitative dynamical Mordell-Lang conjecture}

\author{Alina Ostafe}
\address{School of Mathematics and Statistics, University of New South Wales, Sydney NSW 2052, Australia}
\email{alina.ostafe@unsw.edu.au}

\author{Min Sha}
\address{School of Mathematics and Statistics, University of New South Wales, Sydney NSW 2052, Australia}
\email{shamin2010@gmail.com}


\begin{abstract}
The dynamical Mordell-Lang conjecture concerns the structure of the intersection of an orbit in an algebraic dynamical system and an algebraic variety. In this paper, we bound the size of this intersection for various cases when it is finite.
\end{abstract}

\maketitle

\section{Introduction}

\subsection{Motivation}

Let $\X$ be an algebraic variety defined over the complex numbers $\C$, and let $\Phi: \X \to \X$ be a morphism. For any integer $n\ge 0$, we denote by $\Phi^{(n)}$ the $n$-th iteration of $\Phi$ with $\Phi^{(0)}$ denoting the identity map.

Throughout the paper, a single integer is viewed as an arithmetic progression with common difference 0.

The following is the well-known \textit{dynamical Mordell-Lang conjecture} for self-morphisms of algebraic varieties in the dynamical setting; see \cite{Denis,GT,GTZ1}.

\begin{conjecture}[Dynamical Mordell-Lang Conjecture]
\label{ML}
Let $\X$ and $\Phi$ be given as the above, let $V\subseteq \X$ be a closed  subvariety, and let $P\in \X(\C)$. Then, the following subset of integers
$$
\{n\ge 0: \Phi^{(n)}(P)\in V(\C)\}
$$
is a finite union of arithmetic progressions.
\end{conjecture}

Conjecture \ref{ML} has been studied extensively in recent years. However, so far there are only a few related results.
These include results on maps of various special types  \cite{Bell,BGT,BGKT,Fak,GT,GTZ1,SiVi,Xie} (especially diagonal maps), and analogues for Noetherian spaces \cite{BGT14} and Drinfeld modules \cite{GT2008}.

Recently, Silverman and Viray~\cite[Corollary 1.4]{SiVi} have given results regarding the uniform boundedness (only in terms of $m$) of intersections of orbits of the power map (with the same exponent) at a point of the projective $m$-space $\PP^m(\C)$ with non-zero multiplicatively independent coordinates, with any linear subspace of $\PP^m(\C)$. However, they have not provided quantitative results.
In fact, such a result follows, even in a more general case, directly from the uniform bound on the number of zeros of simple and non-degenerate linear recurrence sequences.

We also note that the uniform boundedness condition has recently been  considered in~\cite{DOSS}, where several results are given for the frequency of the points in an
orbit  of an algebraic dynamical system that belong to a given
  algebraic variety under the reduction modulo a prime $p$.

\subsection{Our Results}
In this paper, we study the quantitative version of Conjecture \ref{ML} for polynomial morphisms of several special types when $\X$ is the affine $m$-space $\A^m(\C)$ and $V$ is a hypersurface; see Section \ref{sec:quan}.  Our main objective is to find as many classes of polynomial morphisms as possible having uniform bounds (or as close as possible to uniformity), and not to investigate detailedly the quality of the bounds.
To the best of our knowledge, this is the first work on the quantitative dynamical Mordell-Lang conjecture.

Here, we extend the results of Silverman and Viray~\cite{SiVi} in two directions. First, we consider monomial systems with different exponents. Second, we estimate the size of the intersection of an orbit with a hypersurface rather than with a hyperplane.

For example, we illustrate a typical case of our results; see Theorem \ref{thm:power d} for more details.
Let $\Phi=\(X_1^d,\ldots,X_m^d\)$ with integer $d\ge 2$ be a diagonal endomorphism of $\A^m(\C)$. Fix a hypersurface $V$, defined by a non-zero polynomial
$$
G=\sum_{i_1,\ldots,i_m}a_{i_1,\ldots,i_m}X_1^{i_1}\cdots X_m^{i_m}\in\C[X_1,\ldots,X_m].
$$
Then, for any $\w\in(\C^*)^m$ with multiplicatively independent coordinates,  the size of the intersection of $V$ and the orbit of $\Phi$ at the point $\w$ is at most
$$
(8\n(G))^{4\n(G)^5},
$$
where $\n(G)$ is the number of monomials of $G$.

Our methods rely on estimates (when finite) for integer solutions of certain polynomial-exponential equations. For the case of the power map studied by Silverman and Viray~\cite{SiVi} we employ results on the number of zeros in linear recurrence sequences due to~\cite{AV1,AV2} and~\cite{PSc}. For more general monomial systems we use results on the number of solutions in a finitely generated subgroup of $\(\C^*\)^k$ of linear equations of the form $a_1x_1+\ldots+a_kx_k=0$, $a_1,\ldots,a_k\in\C^*$, as well as solutions to more general polynomial-exponential equations due to
Schlickewei and Schmidt~\cite{SchSch}.

In fact, by \cite[Theorem 1.8]{GT} the Dynamical Mordell-Lang Conjecture is known to hold in the cases we consider, because the morphisms can  essentially restrict to endomorphisms of $(\C^*)^m$. Besides, the methods we use might be not applicable on other kinds of morphisms, see Section~\ref{sec:com} for more details.

\subsection{Convention and notation}

\label{sec:notation}

For integer $m\ge 2$, let
\begin{equation*}
\Phi = (F_1,\ldots,F_m):\A^m(\C) \to \A^m(\C), \quad F_1,\ldots,F_m \in \C[X_1, \ldots, X_m],
\end{equation*}
be a morphism defined by a system of $m$ polynomials in $m$ variables over $\C$.  For each
$i=1, \ldots ,m$, we define the $n$-th iteration of the polynomials $F_i$ by the recurrence relation
\begin{equation*}
F_i^{(0)}=X_i, \quad F_i^{(n)}= F_i\(F_1^{(n-1)}, \ldots ,F_m^{(n-1)}\), \quad
n=1,2,\ldots,
\end{equation*}
so that
$$
\Phi^{(n)}=\(F_1^{(n)},\ldots,F_m^{(n)}\).
$$
See~\cite{AnKhr,Schm,Silv1} for a background on
dynamical systems associated with such iterations.

For a vector $\vec{w}=(w_1,\ldots,w_m)\in\C^m$, we denote by
$$
\Orb_{\vec{w}}(\Phi)=\left\{\Phi^{(n)}(\vec{w}): n=0,1,2,\ldots\right\}
$$
the orbit of $\Phi$ at  $\vec{w}$.
For an  algebraic variety $V=Z(G_1,\ldots,G_s)$ defined by the equations  $G_1=\cdots=G_s=0$, $G_i\in\C[X_1,\ldots,X_m]$,
$i=1,\ldots,s$, we consider the elements of the orbit $\Orb_{\vec{w}}(\Phi)$ which fall into $V$ and denote
\begin{equation}
\label{eq:intervar}
\S_{\vec{w}}(\Phi,V)
=\left\{ n\ge 0: \Phi^{(n)}(\vec{w})\in V\right\}.
\end{equation}

We say that the complex numbers $\alpha_1,\ldots,\alpha_n$ are \textit{multiplicatively independent} if all of them are non-zero and there is no non-zero integer vector $(i_1,\ldots,i_n)$ such that $\alpha_1^{i_1}\cdots \alpha_n^{i_n}=1$.

In the sequel, we denote by $|S|$ the cardinality of a finite set $S$. Our objective in this paper is to bound the size of $|\S_{\vec{w}}(\Phi,V)|$ for various cases when it is finite.

Throughout the paper, let $\ov{\Q}$ be the algebraic closure of the rational numbers $\Q$. For any field $K$, we write $K^*$ for the multiplicative group of all the non-zero elements of $K$. For any multiplicative group $\Lambda$ and any integer $k\ge 1$, let $\Lambda^k$ be the direct product consisting of $k$-tuples $\vec{x}=(x_1,\ldots,x_k)$ with $x_i\in \Lambda, 1\le i \le k$. As usual,  the multiplication of the group $\Lambda^k$ is defined by  $\vec{x}\vec{y}=(x_1y_1,\ldots,x_ky_k)$ for any $\vec{x},\vec{y}\in \Lambda^k$.

\section{Preliminaries}

In this section, we gather some results which are used afterwards.

Recall that a linear recurrence sequence (LRS) of order $m\ge 1$ is a sequence $\{u_0,u_1,u_2,\ldots\}$ with elements in $\C$ satisfying a linear relation
\begin{equation}
\label{sequence}
u_{n+m}=a_1u_{n+m-1}+\cdots+a_mu_n \quad (n=0,1,2,\ldots),
\end{equation}
where $a_1,\dots,a_m\in \C$, $a_m\ne 0$ and $u_j \ne 0$ for at least one
 $j$ in the range $0 \le j \le m-1$. We assume that relation \eqref{sequence} is minimal, that is the sequence $\{u_n\}$ does not satisfy a relation of type \eqref{sequence} of smaller length.

The characteristic polynomial of this LRS $\{u_n\}$ is
$$
f(X)=X^m-a_1X^{m-1}-\cdots-a_m=\prod_{i=1}^{k}(X-\alpha_i)^{e_i}  \in \C[X]
$$
with distinct $\alpha_1,\alpha_2,\ldots,\alpha_k$ and $e_i>0$ for $1\le i \le k$. Then, $u_n$ can be expressed as
$$
u_n=\sum_{i=1}^{k}f_i(n)\alpha_i^n,
$$
where $f_i$ is some polynomial of degree $e_i-1$ ($i=1,2,\ldots,k$). We call the sequence $\{u_n\}$ \textit{simple} if $k=m$ (that is $e_1=\cdots =e_m=1$) and \textit{non-degenerate} if $\alpha_i/\alpha_j$ is not a root of unity for any $i\ne j$ with $1\le i,j\le k$.

One fundamental problem of the LRS \eqref{sequence} is to describe the structure or bound the size of the following set
$$
\{n\ge 0: u_n=0\},
$$
which is called the \textit{zero set} of the sequence \eqref{sequence}.
Equivalently, we want to study the integer roots of the exponential polynomial $\sum_{i=1}^{k}f_i(z)\alpha_i^z$.

The well-known Skolem-Mahler-Lech theorem says that the zero set of any LRS is a finite union of arithmetic progressions, and furthermore it is a finite set if the sequence is non-degenerate; for example see \cite[Theorem 2.1]{BK}. There are rich results on the quantitative version of this theorem. Here we restate some results in the setting of exponential polynomials, which are used later on.

In the rest of this section, we fix an exponential polynomial over $\C$
\begin{equation}
\label{pol1}
F(z)=\sum_{i=1}^{k}f_i(z)\alpha_i^z
\end{equation}
with distinct $\alpha_1,\ldots,\alpha_k\in \C^*$, and non-zero $f_i\in \C[z]$ for $1\le i\le k$. We also define
$$m=\deg f_1+\cdots+\deg f_k+k$$
 and denote
$$
\cZ(F)=\{n\in \Z: F(n)=0, n\ge 0\}.
$$
Note that $F(z)$ corresponds to an LRS of order $m$, and the set $\cZ(F)$ is exactly the zero set of the corresponding sequence. Especially, when $f_1,\ldots,f_k$ are constants, $F(z)$ corresponds to a simple LRS.

The following result comes from \cite[Corollary 6.3]{AV1} and \cite[Theorem 1.1]{AV2}.

\begin{lemma}
\label{LRS1}
Let $F(z)$ be given by \eqref{pol1}. Then the set $\cZ(F)$ is the union of at most
$\exp(\exp(70m))$
 arithmetic progressions.
Moreover, if $f_1,\ldots,f_k$ are non-zero constants, then
the set $\cZ(F)$ is the union of at most
$(8m)^{4m^5}$ arithmetic progressions.
\end{lemma}

Lemma \ref{LRS1} can yield some quantitative results concerning Conjecture \ref{ML}. However, here we are more interested with the case when the subset of integers in Conjecture \ref{ML} is a finite set.

As mentioned above, if $F(z)$ corresponds to a non-degenerate LRS, the set $\cZ(F)$ is in fact a finite set, and furthermore we can bound the cardinality $|\cZ(F)|$. The following result follows from \cite[Corollary 6.3]{AV1} and \cite[Theorem 1.2]{AV2}.

\begin{lemma}
\label{LRS2}
Let $F(z)$ be given by \eqref{pol1}. Suppose that $F(z)$ corresponds to a non-degenerate LRS, and $\deg f_i+1\le a$ for $1\le i \le k$. Then we have
$$
|\cZ(F)|\le \(8k^a\)^{8k^{6a}};
$$
furthermore if $f_1,\ldots,f_k$ are non-zero constants, then
we have
$$
| \cZ(F) | \le (8m)^{4m^5}.
$$
\end{lemma}

In fact, if there exists some index $i$ such that the ratio $\alpha_i/\alpha_j$ is not a root of unity for any $j\ne i$, then the set $\cZ(F)$ is still a finite set; see \cite[Theorem 2.1 (iii)]{BK}. Here we want to bound $|\cZ(F)|$ in this case by using Lemma \ref{LRS2} and following the arguments in \cite{BK}.

\begin{lemma}
\label{LRS3}
Let $F(z)$ be given by \eqref{pol1}. Let $D$ be the order of the group of roots of unity generated by all those roots of unity which are of the form $\alpha_i/\alpha_j$ for some $1\le i,j \le k$. Suppose that there exists some index $i_0$ such that the ratio $\alpha_{i_0}/\alpha_j$ is not a root of unity for any $j\ne i_0$, and $\deg f_i+1\le a$ for $1\le i \le k$. Then we have
$$
| \cZ(F) | \le D\(8k^a\)^{8k^{6a}}.
$$
\end{lemma}
\begin{proof}
We partition $\alpha_1,\ldots,\alpha_k$ into equivalence classes according to the equivalence relation where $b\sim c$ if and only if the ratio $b/c$ is a root of unity. By renumbering, we can assume that $\alpha_1,\ldots,\alpha_s$ are representatives of these equivalence classes. Then, fixing an  integer $b$ with $0\le b <D$, we consider the equation
$$F(b+nD)=0$$
 with integer unknown $n\ge 0$. By the choice of $D$, we can express $F(b+nD)$ as
 $$
 F(b+nD)=\sum_{i=1}^{s}g_i(n)\(\alpha_i^D\)^n
 $$
 for some polynomials $g_i\in \C[z]$ with $\deg g_i+1\le a$. Under the assumption on $\alpha_{i_0}$, there indeed exists some index $j$ such that $g_j\ne 0$. So, using Lemma \ref{LRS2}, we deduce that the cardinality of the set $\{n\ge 0:F(b+nD)=0\}$ is at most $\(8k^a\)^{8k^{6a}}$. The final result follows from the fact that there are $D$ choices of the integer $b$.
\end{proof}

Note that if $F(z)$ corresponds to a non-degenerate sequence, then $D=1$ and Lemma \ref{LRS3} is exactly the first upper bound in Lemma \ref{LRS2}.

Moreover, if $\alpha_1,\ldots,\alpha_k$ are roots of a polynomial $f(X)$ over a number field $K$, then the quantity $D$ can be bounded by
\begin{equation}
\label{Dubickas}
D < \exp \( \(1.05314 +\sqrt{6d}\)\sqrt{m \log(dm)} \),
\end{equation}
where $d=[K:\Q]$ and $m=\deg f \ge 2$; see \cite[Theorem 1]{Dubickas2012}.

Except for studying the set $\cZ(F)$, we also need to estimate the number of integers $n\ge 0$ such that $F(n)$ is equal to a fixed non-zero complex number.

\begin{corollary}
\label{cor:LRS}
Let $F(z)$ be given by \eqref{pol1}. Define $\alpha_{k+1}=1$. Suppose that there exists some index $i_0$ such that the ratio $\alpha_{i_0}/\alpha_j$ is not a root of unity for any $j\ne i_0$ with $1\le j \le k+1$. Let $D$ be the order of the group of roots of unity generated by all those roots of unity which are of the form $\alpha_i/\alpha_j$ for
some $1\le i,j \le k+1$.
Assume that $\deg f_i+1\le a$ for $1\le i \le k$. Then for any $\mu\in \C$ with $\mu \ne 0$, we have
$$
| \{ n\in \Z: F(n)=\mu, n\ge 0\} | \le D\(8(k+1)^a\)^{8(k+1)^{6a}};
$$
furthermore if $F(z)$ corresponds to a non-degenerate LRS, no $\alpha_i$ ($1\le i \le k$) is a root of unity, and $f_1,\ldots,f_r$ are non-zero constants, we have
$$
| \{ n\in \Z: F(n)=\mu, n\ge 0\} | \le (8(m+1))^{4(m+1)^5}.
$$
\end{corollary}

\begin{proof}
Under the assumptions, we can get the desired results by applying Lemmas \ref{LRS2} and \ref{LRS3} to the following equation
$$
\sum_{i=1}^{k}f_i(n)\alpha_i^n+(-\mu)\cdot 1=0, \quad n \ge 0,
$$
with coefficients $f_1(n),\ldots,f_k(n),-\mu$.
\end{proof}

When $\alpha_1,\ldots,\alpha_k$ are algebraic numbers, the results in Lemma \ref{LRS3} and Corollary \ref{cor:LRS} can be improved in some sense. The following lemma is derived from \cite[Theorem 1]{PSc} with a slight refinement. Although the arguments in \cite{PSc} were presented only for non-degenerate sequences, they are also valid for more general cases under minor changes.

\begin{lemma}
\label{LRS4}
Let $F(z)$ be given by \eqref{pol1}. Suppose that $\alpha_1,\ldots,\alpha_k$ are algebraic numbers, and let $D$ be the order of the group of roots of unity generated by all those roots of unity which are of the form $\alpha_i/\alpha_j$ for some $1\le i,j \le k$. Denote $K=\Q(\alpha_1,\ldots,\alpha_k)$, and assume that $f_1,\ldots,f_k\in K[z]$. Let $d=[K:\Q]$, and let $\omega$ be the number of prime ideals occurring in the decomposition of the fractional ideals $\gen{\alpha_i}$ in $K$. Assume  that there exists some index $i_0$ such that the ratio $\alpha_{i_0}/\alpha_j$ is not a root of unity for any $j\ne i_0$. Then, we have
$$
| \cZ(F) | < D(4(d+\omega))^{2(d+1)}(m-1);
$$
furthermore if $K/\Q$ is a Galois extension but not a cyclic extension, we have
$$
| \cZ(F) | < D(4(d+\omega))^{d+2}(m-1).
$$
\end{lemma}

\begin{proof}
Here, we sketch the proof for the convenience of the reader.

We first choose a rational prime $p$ such that none of the prime ideals $\p_1,\ldots,\p_\omega$ from the decomposition in $K$ of the ideals $(\alpha_i)$ ($1\le i \le k$) divides the ideal $\gen{p}$. Let $\p$ be a prime ideal of $K$ lying above $p$, and let $f$ denote the residue class degree of $K_\p$ over $\Q_p$, where $\Q_p$ is the $p$-adic completion of $\Q$ and $K_\p$ is the completion of $K$ with respect to $\p$. Let $\C_p$ be the completion of the algebraic closure of $\Q_p$. We denote the valuation of $\C_p$ by $|\,\,|_p$, normalised such that $|p|_p=p^{-1}$. Note that $\Q_p \subseteq K_\p \subseteq \C_p$.

By the choice of $p$, we have
$$
|\alpha_i|_p=1, \quad i=1,\ldots,k.
$$
Furthermore, by \cite[Equation (3.4)]{PSc} we know
$$
\left|\alpha_i^{p^f-1}-1\right|_p<p^{-1/(p(p-1))-1/(p-1)}, \quad i=1,\ldots,k.
$$
Then, we have
$$
\left|\alpha_i^{D(p^f-1)}-1\right|_p\le \left|\alpha_i^{p^f-1}-1\right|_p <p^{-1/(p(p-1))-1/(p-1)}, \quad i=1,\ldots,k.
$$
Fix an integer $a$ with $0\le a < D(p^f-1)$, we consider the equation
$$F\(a+nD\(p^f-1\)\)=0$$
 with integer unknown $n\ge 0$.

As in the proof of Lemma \ref{LRS3}, we partition $\alpha_1,\ldots,\alpha_k$ into equivalence classes, and we assume that $\alpha_1,\ldots,\alpha_s$ are representatives of these equivalence classes. Then, by the choice of $D$, we can express $F\(a+nD\(p^f-1\)\)$ as
 $$
 F\(a+nD\(p^f-1\)\)=\sum_{i=1}^{s}g_i(n)\(\alpha_i^{D(p^f-1)}\)^n
 $$
 for some polynomials $g_i\in K[z]$. Under the assumption of $\alpha_{i_0}$, there indeed exists some index $j$ such that $g_j\ne 0$.

 As solving the equation (3.6) of \cite{PSc}, we immediately see that the cardinality of the set $\left\{n\in \Z:F\(a+nD\(p^f-1\)\)=0, n\ge 0\right\}$ is at most $(m-1)(p+1)$. Thus, we obtain
 $$
 | \cZ(F) | \le D\(p^f-1\)(m-1)(p+1).
 $$
 From \cite[Section 4]{PSc}, the prime $p$ can be chosen such that $$p<(4(d+\omega))^2.$$
Then, the first desired upper bound follows from the fact that $f\le d$.

 Now, we assume that $K/\Q$ is a Galois extension but not a cyclic extension.
 In order to prove the second claimed upper bound, it suffices to show that $p$ does not remain inert in $K$. Because if this is true, then $f\le d/2$, which can conclude the proof.

Let $D_p$  denote the decomposition group of $p$ in $K/\Q$. Suppose that $p$ remains inert in $K$. Then, $f=d$, and $D_p$ is a cyclic group of order $d$. Since $[K:\Q]=d$, $D_p$ is exactly the Galois group $\Gal(K/\Q)$. So, $K/\Q$ is a cyclic extension, this leads to a contradiction.
\end{proof}

Applying the same arguments as in the proof of Corollary \ref{cor:LRS}, we can obtain similar results as Lemma \ref{LRS4} for the cardinality $| \{ n\in \Z: F(n)=\mu, n\ge 0\} |$, where $\mu$ is a non-zero algebraic number.

We also need a result on solutions of linear equations in several variables. The following result is given in \cite[Lemma 2.1]{AV2} and is derived from \cite[Theorem 6.2]{AV1}.

\begin{lemma}
\label{thm:Sunit}
Let $\Gamma$ be finitely generated subgroup of $(\C^*)^k$ of rank $r$, and let $a_1,\ldots,a_k\in \C^*$. Then, up to proportionality, the equation
\begin{equation}
\label{eq:unit}
a_1x_1+\cdots+a_kx_k=0
\end{equation}
has less than $(8k)^{4(k-1)^4(k+r)}$ non-degenerate solutions in $\Gamma$.
\end{lemma}

Here, ``\textit{up to proportionality}'' means that two 
solutions $(x_1,\ldots,x_k)$ and $(y_1,\ldots,y_k)$ of \eqref{eq:unit} are equivalent if there is some non-zero $c$ such that
$$
  (y_1,\ldots,y_k) = (c x_1,\ldots,c x_k).
$$
Besides, we call a solution of \eqref{eq:unit} \textit{non-degenerate} if no subsum of the left hand side of the equation vanishes. 

Let $K$ be a number field, let $S$ be a finite set of places of $K$ containing all the Archimedean places and write $\cO_S^*$ for the group of $S$-units of $K$.
If the coefficients $a_1,\ldots,a_k\in K \setminus \{0\}$, then the number of such solutions of \eqref{eq:unit} in $\Gamma \subseteq (\cO_S^*)^k$ can be bounded better than Lemma \ref{thm:Sunit}; for example see \cite[Theorem 3]{Ever}. So, some results in this paper can be improved in this case.

Let $\cP$ be a partition of the set $I=\{1,\ldots,k\}$. The subsets $\lambda\subseteq I$ occurring in the partition $\cP$ are considered as elements of $\cP$. Then, the system of equations
\begin{equation}
\label{eqs:Sunit}
\sum_{i\in \lambda} a_ix_i=0 \quad (\lambda\in \cP)
\tag{2.4 $\cP$}
\end{equation}
is a refinement of \eqref{eq:unit}. Given a partition $\cP$ of the set $I$, we say that two 
solutions $(x_1,\ldots,x_k)$ and $(y_1,\ldots,y_k)$ of \eqref{eq:unit}  are equivalent \textit{up to proportionality with respect to $\cP$} if both of them are also solutions of the system \eqref{eqs:Sunit}, and for each $\lambda\in \cP$ the two solutions $(x_i)_{i\in \lambda}$ and $(y_i)_{i\in \lambda}$ of the corresponding equation are equivalent up to proportionality. 

Finally, two 
solutions $(x_1,\ldots,x_k)$ and $(y_1,\ldots,y_k)$ of \eqref{eq:unit}  are called equivalent \textit{up to weak proportionality} if there exists a partition $\cP$ of the set $I$ such that they are equivalent up to proportionality with respect to $\cP$.

Now, we want to count all the solutions of \eqref{eq:unit} up to weak  proportionality.

\begin{corollary}
\label{cor:Sunit}
Let $\Gamma$ be finitely generated subgroup of $(\C^*)^k$ of rank $r$, and let $a_1,\ldots,a_k\in \C^*$. Then, up to weak  proportionality,  the equation \eqref{eq:unit}
has less than
$$(0.5k)^k(8k)^{4(k-1)^4(k+r)}$$
 solutions in $\Gamma$.
\end{corollary}

\begin{proof} 
Let $\cP$ be a partition of the set $I=\{1,\ldots,k\}$.
 Note that in order to ensure that the system \eqref{eqs:Sunit} has solutions in $\Gamma$ we must have that $|\lambda|\ge 2$ for any $\lambda\in \cP$. So, we can assume that $|\cP|\le k/2$.

If $\cQ$ is another partition of $I$ such that $\cQ$ is a refinement of $\cP$, then the system (2.4 $\cQ$) implies \eqref{eqs:Sunit}. Let $\cT(\cP)$ consist of solutions of \eqref{eqs:Sunit} in $\Gamma$ up to proportionality with respect to $\cP$, which are not solutions of any (2.4 $\cQ$) where $\cQ$ is a proper refinement of $\cP$.

According to the partition $\cP$, we can treat $\Gamma$ as a direct product
$$
\Gamma= \prod_{\lambda\in \cP} \Gamma(\lambda),
$$
where $\Gamma(\lambda)$ is the projection of $\Gamma$ corresponding to $\lambda$. For each $\lambda \in \Gamma$, $\Gamma(\lambda)$ is also a finitely generated group and let $r(\lambda)$ be its rank. Obviously, we have
$$
\sum_{\lambda\in \cP} r(\lambda) = r.
$$

For each equation in \eqref{eqs:Sunit}
$$
\sum_{i\in \lambda} a_ix_i=0,
$$
by Lemma \ref{thm:Sunit} it has less than $(8|\lambda|)^{4(|\lambda|-1)^4(|\lambda|+r(\lambda))}$ non-degenerate solutions in $\Gamma(\lambda)$ up to proportionality. Thus, we have
\begin{align*}
|\cT(\cP)| & < \prod_{\lambda\in \cP} (8|\lambda|)^{4(|\lambda|-1)^4(|\lambda|+r(\lambda))}\\
& < (8k)^{4(k-1)^4(k+r)}.
\end{align*}

Recall that the Bell numbers count the number of partitions of a set. By \cite[Theorem 2.1]{BT}, the number of partitions of $I$ is less than
$$
\left( 0.792k/\log (k+1) \right)^k.
$$
However, not all such partitions are suitable for our settings. We have indicated that a \textit{suitable partition} $\cP$ should satisfy that $|\lambda|\ge 2$ for any $\lambda\in \cP$. 
So, the number of these suitable partitions is not greater than 
$\left( 0.5k \right)^k$. 

Note that every solution of the equation \eqref{eq:unit} is a solution of the system \eqref{eqs:Sunit} for some partition $\cP$, and we can assume that $k\ge 2$.
So, up to weak proportionality, the number of solutions of \eqref{eq:unit} in $\Gamma$ is at most
\begin{align*}
\sum_{\cP} |\cT(\cP)|&< \left( 0.5k \right)^k  (8k)^{4(k-1)^4(k+r)},
\end{align*}
where the sum runs through all suitable partitions of $I$. 
This completes the proof.
\end{proof}

Finally, we state a result due to
Schlickewei and Schmidt~\cite[Theorem 1]{SchSch}. For a vector $\vec{x}=(x_1,\ldots,x_m)\in\Z^m$ and $\pmb{\alpha}=(\alpha_1,\ldots,\alpha_m)\in\ov{\Q}^m$, we denote
$$
\pmb{\alpha}^{\vec{x}}=\alpha_{1}^{x_1}\cdots \alpha_m^{x_m}.
$$

\begin{lemma}
\label{thm:schsch}
Fix $\pmb{\alpha}_i\in (\ov{\Q}^*)^m$, $i=1,\ldots,k$, such that for any $1\le i\ne j\le k$ the set of $\vec{z}\in\Z^m$ with
$$
\pmb{\alpha}_i^{\vec{z}}=\pmb{\alpha}_j^{\vec{z}}
$$
 contains only the zero vector. Then, the number of solutions to the equation
$$
\sum_{i=1}^k a_i \pmb{\alpha}_i^{\vec{x}}=0
$$
with non-zero algebraic numbers $a_i$
is less than
$$(0.5k)^k2^{35B^3}d^{6B^2},$$
where $B=\max(m,k)$ and $d$ is the degree of the number field generated by the coefficients $a_i$ and the vectors $\pmb{\alpha}_i$.
\end{lemma}

\begin{proof}
The desired result can be easily proved by using \cite[Theorem 1]{SchSch} and counting the solutions through all the partitions of the set $\{1,2,\ldots,k\}$.
\end{proof}

\section{Main results}
\label{sec:quan}

We first recall a notation. For any polynomial $G\in\C[X_1,\ldots,X_m]$, we let
$$
\n(G)=\textrm{the number of monomials of $G$}.
$$
As usual, here we treat the non-zero constant term as a monomial.

Recall that for any point $\w\in \C^m$, we denote its coordinates by $(w_1,\ldots,w_m)$. We start the discussions by dealing with the simplest case.

\begin{theorem}
\label{thm:power d}
Let $\Phi=\(X_1^d,\ldots,X_m^d\)$ with integer $d\ge 2$. Fix a hypersurface $V=Z(G)$, where
$$
G=\sum_{i_1,\ldots,i_m}a_{i_1,\ldots,i_m}X_1^{i_1}\cdots X_m^{i_m}\in\C[X_1,\ldots,X_m]
$$
with $G\ne 0$ and $i_j\ge 0$ for $1\le j \le m$.
Then, for any $\w\in(\C^*)^m$ with multiplicatively independent coordinates,  we have
$$
| \S_{\vec{w}}(\Phi,V) | \le (8\n(G))^{4\n(G)^5}.
$$
\end{theorem}
\begin{proof}
For the given point $\vec{w}$, we want to bound the number of integers $n\ge 0$ such that $\Phi^{(n)}(\vec{w})\in V$, that is
$$
\sum_{i_1,\ldots,i_m}a_{i_1,\ldots,i_m}\(w_1^{i_1} \cdots w_m^{i_m}\)^{d^n}=0.
$$
This is upper bounded by the number of integers $n\ge 0$ such that
$$
\sum_{i_1,\ldots,i_m}a_{i_1,\ldots,i_m}\(w_1^{i_1} \cdots w_m^{i_m}\)^{n}=0.
$$
Since the coordinates of $\vec{w}$ are multiplicatively independent, the upper bound follows from Lemma \ref{LRS2} and Corollary \ref{cor:LRS} by noticing whether $G$ has a non-zero constant term or not.
\end{proof}

One can relax the multiplicative independence condition on the point $\w$ in some special cases.

\begin{theorem}
\label{thm:pd}
Let $\Phi=\(X_1^d,\ldots,X_m^d\)$ with integer $d\ge 2$. Fix a hypersurface $V=Z(G)$, where
$$
G=\sum_{j=1}^{m}a_jX_j^{e_j}\in\C[X_1,\ldots,X_m]
$$
with $G\ne 0$ and $e_j\ge 0$ for $1\le j \le m$.
For any point $\w \in \C^m$, let $D$ be the order of the group of roots of unity generated by all those roots of unity which are of the form $w_i/w_j$ for some $1\le i,j \le m$. Suppose that there exists some index $j_0$ such that $w_{j_0}$ is not a root of unity, $w_{j_0}\ne 0, a_{j_0}\ne 0, e_{j_0}\ne 0$ and the ratio $w_j/w_{j_0}$ is not a root of unity for any $j\ne j_0$. Then, we have
$$
| \S_{\w}(\Phi,V) | \le D(8\n(G))^{8\n(G)^6}.
$$
\end{theorem}
\begin{proof}
Applying the same arguments as in the proof of Theorem \ref{thm:power d}, the desired result follows directly from Lemma \ref{LRS3} and Corollary \ref{cor:LRS}.
\end{proof}

The upper bound in Theorem \ref{thm:pd} is not uniform because of the quantity $D$. However, we can make it uniform in some sense. In fact, if we choose the point $\w \in K^m$, where $K$ is a number field, then  $D$ does not exceed the number of roots of unity contained in $K$. Alternatively, one can also use \eqref{Dubickas}.

We also want to indicate that in Theorem \ref{thm:pd} if we further assume that $\vec{w}\in \ov{\Q}^m$ and $G\in \ov{\Q}[X_1,\ldots,X_m]$, then under the same assumptions as in Lemma \ref{LRS4}, we can get another upper bound for $|\S_{\vec{w}}(\Phi,V)|$.

Obtaining results on the size of $\S_{\w}(\Phi,V)$ even for the slightly more general case when $F_i=X_i^{d_i}$, $i=1,\ldots,m$, where the degrees $d_i$ are not necessarily the same, seems not to be quite straightforward.

\begin{theorem}
\label{thm:power di}
Let $\Phi=\(X_1^{d_1},\ldots,X_m^{d_m}\)$ with integers $d_i\ge 2$ ($1\le i \le m$). Fix a hypersurface $V=Z(G)$, where
$$
G=\sum_{i_1,\ldots,i_m}a_{i_1,\ldots,i_m}X_1^{i_1}\cdots X_m^{i_m}\in\C[X_1,\ldots,X_m]
$$
with $G\ne 0$ and $i_j\ge 0$ for $1\le j \le m$.
Then, for any $\vec{w}\in(\C^*)^m$ with multiplicatively independent coordinates,  we have
$$
| \S_{\vec{w}}(\Phi,V) | \le (0.5\n(G))^{\n(G)}(8\n(G))^{4\n(G)(\n(G)-1)^4(m+1)}.
$$
\end{theorem}

\begin{proof}
Since the polynomial $G$ has $\n(G)$ monomials, we renumber the indices $(i_1,\ldots,i_m)$ as $1,2,\ldots, \n(G)$. If the index $(i_1,\ldots,i_m)$ corresponds to $j$ ($1\le j \le \n(G)$), then accordingly we write $a_{i_1,\ldots,i_m}$ and $X_1^{i_1}\cdots X_m^{i_m}$ as $b_j$ and $Y_j$, respectively. So, we have
$$
G=\sum_{j=1}^{\n(G)} b_jY_j.
$$

For the given point $\w$, bounding $|\S_{\vec{w}}(\Phi,V)|$ is exactly to bound the number of integers $n\ge 0$ such that $\Phi^{(n)}(\vec{w})\in V$, that is,
\begin{equation}
\label{eq:general1}
\sum_{j=1}^{\n(G)} b_jY_j\(\Phi^{(n)}(\vec{w})\)=0.
\end{equation}

Let $\Lambda$ be the group generated by the coordinates of $\w$, and let $\Gamma=\Lambda^{\n(G)}$. Since the rank of $\Lambda$ is $m$, the rank of $\Gamma$ is at most $m\n(G)$. In view of \eqref{eq:general1}, we consider the solutions $(x_1,\ldots,x_{\n(G)})$ of the equation
\begin{equation}
\label{eq:general2}
b_1x_1+\cdots + b_{\n(G)}x_{\n(G)}=0
\end{equation}
in $\Gamma$. By Corollary \ref{cor:Sunit}, the equation \eqref{eq:general2} has less than
$$
(0.5\n(G))^{\n(G)}(8\n(G))^{4(\n(G)-1)^4(\n(G)+m\n(G))}
$$
solutions in $\Gamma$ up to weak proportionality.

For any $1\le i \ne j \le \n(G)$, write $Y_i=X_1^{i_1}\cdots X_m^{i_m}$ and $Y_j=X_1^{j_1}\cdots X_m^{j_m}$. Suppose that there exist integers $n,k$ with $0\le n<k$ such that
$$
\frac{Y_i\(\Phi^{(k)}(\vec{w})\)}{Y_i\(\Phi^{(n)}(\vec{w})\)}=\frac{Y_j\(\Phi^{(k)}(\vec{w})\)}{Y_j\(\Phi^{(n)}(\vec{w})\)}.
$$
Then due to $\Phi^{(n)}=\(X_1^{d_1^n},\ldots,X_m^{d_m^n}\)$, we obtain
$$
w_1^{i_1(d_1^k-d_1^n)}\cdots w_m^{i_m(d_m^k-d_m^n)}=w_1^{j_1(d_1^k-d_1^n)}\cdots w_m^{j_m(d_m^k-d_m^n)}.
$$
Noticing that $d_\ell \ge 2$ ($1\le \ell \le m$) and the coordinates $w_1,\ldots,w_m$ are multiplicatively independent, we must have $i_1=j_1,\ldots, i_m=j_m$, which implies that $Y_i=Y_j$. This is a contradiction with $Y_i\ne Y_j$.

Hence, for any $1\le i \ne j \le \n(G)$ and any $0\le n < k$, we have 
$$
\frac{Y_i\(\Phi^{(k)}(\vec{w})\)}{Y_i\(\Phi^{(n)}(\vec{w})\)}\ne \frac{Y_j\(\Phi^{(k)}(\vec{w})\)}{Y_j\(\Phi^{(n)}(\vec{w})\)}.
$$

Thus, the number of those solutions of \eqref{eq:general2} with the form
\begin{equation}
\label{eq:solution}
\(Y_1\(\Phi^{(n)}(\vec{w})\),\ldots, Y_{\n(G)}\(\Phi^{(n)}(\vec{w})\)\) \quad (n=0,1,2,\ldots)
\end{equation}
is less than
$$
(0.5\n(G))^{\n(G)}(8\n(G))^{4\n(G)(\n(G)-1)^4(m+1)}.
$$
Notice that there is a one-to-one correspondence between the vectors \eqref{eq:solution} and the integers $n\ge 0$, we complete the proof.
\end{proof}

Now, we want to use Lemma \ref{thm:schsch} to give another method on  handling a special case in Theorem \ref{thm:power di}.

\begin{theorem}
\label{thm:pdi2}
Let $K$ be a number field of degree $d$. Let the system $\Phi=\(X_1^{d_1},\ldots,X_m^{d_m}\)$ with integers $d_i\ge 0$ ($1\le i \le m$) and some index $\ell$ such that $d_\ell\ge 2$. Fix a hypersurface $V=Z(G)$, where
$$
G=\sum_{i_1,\ldots,i_m}a_{i_1,\ldots,i_m}X_1^{i_1}\cdots X_m^{i_m}\in K[X_1,\ldots,X_m]
$$
with $G\ne 0$ and $i_j\ge 0$ for $1\le j \le m$,
such that $G$ has zero constant term. Suppose that for any two monomials $a_{i_1,\ldots,i_m}X_1^{i_1}\cdots X_m^{i_m}$ and $a_{j_1,\ldots,j_m}X_1^{j_1}\cdots X_m^{j_m}$ of $G$, we have $i_1\ne j_1, \ldots, i_m\ne j_m$.
Then, for any $\w \in (K^*)^m$ with multiplicatively independent coordinates,  we have
$$
| \S_{\vec{w}}(\Phi,V) | \le (0.5\n(G))^{\n(G)}2^{35B^3}d^{6B^2},
$$
where $B=\max(m,\n(G))$.
\end{theorem}

\begin{proof}
Since $\Phi^{(n)}=\(X_1^{d_1^n},\ldots,X_m^{d_m^n}\)$, we need to
bound the number of integers $n\ge 0$ such that
$$
\sum_{i_1,\ldots,i_m}a_{i_1,\ldots,i_m}w_1^{i_1d_1^n}\cdots w_m^{i_md_m^n}=0.
$$
For each index $(i_1,\ldots,i_m)$, we write $\pmb{\alpha}_i=\(w_1^{i_1},\ldots,w_m^{i_m}\)$ and $\vec{x}=\(d_1^n,\ldots,d_m^n\)$, then the above equation becomes
\begin{equation}
\label{eq:exp}
\sum_{i_1,\ldots,i_m}a_{i_1,\ldots,i_m}\pmb{\alpha}_i^{\vec{x}}=0.
\end{equation}
Moreover, for any two indices $(i_1,\ldots,i_m)$ and $(j_1,\ldots,j_m)$, if $\pmb{\alpha}_i^\vec{z}=\pmb{\alpha}_j^\vec{z}$ for $\vec{z}=(z_1,\ldots,z_m)\in \Z^m$,
then, we must have $i_1z_1=j_1z_1, \ldots, i_mz_m=j_mz_m$ by using the multiplicative independence condition. Under the assumption that $i_1\ne j_1,\ldots, i_m\ne j_m$, we get $z_1=\cdots=z_m=0$. That is, $\vec{z}$ is the zero vector.

Now, applying Lemma \ref{thm:schsch} to the equation \eqref{eq:exp} we know that
the equation \eqref{eq:exp} has less than
$$
(0.5\n(G))^{\n(G)}2^{35B^3}d^{6B^2}
$$
solutions with  the form $(d_1^n,\ldots,d_m^n), n=0,1,2,\ldots$.
Besides, since $d_\ell \ge 2$, we have $\(d_1^n,\ldots,d_m^n\) \ne \(d_1^k,\ldots,d_m^k\)$ for any integers $n\ne k$. So, the desired result follows.
\end{proof}

It seems natural to expect that the classes of dynamical systems and hypersurfaces that satisfy the uniform boundedness condition are quite wide. We confirm this by the following three theorems.

\begin{theorem}
\label{thm:pdi3}
Let  $\Phi=\(X_1^{d},F_2,F_3,\ldots,F_m\)$ with integer $d\ge 2$, where
$$
F_i=X_1^{s_{i1}}\cdots X_m^{s_{im}}
$$
with $s_{ij}\ge 0,s_{ii}\ge 1, 2\le i \le m, 1\le j \le m$, such that $ \deg F_i < d$ for any $2\le i \le m$. 
Fix a hypersurface $V=Z(G)$, where
$$
G=aX_1^{e}+\sum_{i_1,\ldots,i_m}a_{i_1,\ldots,i_m}X_1^{i_1}\cdots X_m^{i_m}\in\C[X_1,\ldots,X_m]
$$
with $a\ne 0,  e\ge 1$ and $i_j\ge 0$ for $1\le j \le m$. Suppose that  $\deg G = e$. 
Then, for any $\vec{w}\in(\C^*)^m$ with multiplicatively independent coordinates,  
we have
$$
| \S_{\vec{w}}(\Phi,V) | \le (0.5\n(G))^{\n(G)}(8\n(G))^{4\n(G)(\n(G)-1)^4(m+1)}.
$$
\end{theorem}

\begin{proof}
As in the proof of Theorem \ref{thm:power di} and under the assumptions of the polynomial $G$, we can write
$$
G=\sum_{j=1}^{\n(G)} b_jY_j
$$
such that $b_1=a,Y_1=X_1^{e}$.
We need to
bound the number of integers $n\ge 0$ such that
$$
\sum_{j=1}^{\n(G)} b_jY_j\(\Phi^{(n)}(\vec{w})\)=0.
$$

For any $n\ge 0$, as in Section \ref{sec:notation} we write $\Phi^{(n)}=(F_1^{(n)}, F_2^{(n)},\ldots, F_m^{(n)})$ with $F_1^{(n)}=X_1^{d^n}$. Since $d> \deg F_i$ for any $2\le i \le m$, we can see that 
\begin{equation}
\label{eq:degree}
\deg F_i^{(n)} < d^n, \quad \textrm{for any $n\ge 1$} .
\end{equation}
Fix $2\le i\le m$. Then, for any $n\ge 0$, since 
$$
\Phi^{(n+1)}=\Phi(F_1^{(n)}, F_2^{(n)},\ldots, F_m^{(n)}) \mand s_{ii} \ge 1,
$$
applying \eqref{eq:degree} we deduce that 
\begin{align*}
& \deg_{X_1} F_i^{(n+1)} - \deg_{X_1} F_i^{(n)} \\
& \quad = \sum_{j=1, j\ne i}^m s_{ij}\deg_{X_1} F_j^{(n)} + (s_{ii}-1) \deg_{X_1} F_i^{(n)}\\
& \quad < d^n \big(\sum_{j=1}^{m} s_{ij} -1 \big) < d^n (d-1).
\end{align*}
So, for any integers $n,k$ with $0\le n<k$, we obtain  
\begin{equation}
\label{eq:degX1}
0\le \deg_{X_1} F_i^{(k)} - \deg_{X_1} F_i^{(n)} < d^k - d^n.
\end{equation}

Note that for any integers $n,k$ with $0\le n<k$, we have 
$$
\frac{Y_1\(\Phi^{(k)}(\vec{w})\)}{Y_1\(\Phi^{(n)}(\vec{w})\)}
=w_1^{e(d^k-d^n)}.
$$
Combining \eqref{eq:degX1} with $\deg G =e$, we can see that 
 for any $2\le j \le \n(G)$, the degree of $Y_j\(\Phi^{(k)}(\vec{w})\)/Y_j\(\Phi^{(n)}(\vec{w})\)$ with respect to $w_1$ is less than $e(d^k-d^n)$. Notice that the coordinates of $\vec{w}$ are multiplicatively independent. Thus, for any $2\le j \le \n(G)$ we have
$$
\frac{Y_1\(\Phi^{(k)}(\vec{w})\)}{Y_1\(\Phi^{(n)}(\vec{w})\)} \ne \frac{Y_j\(\Phi^{(k)}(\vec{w})\)}{Y_j\(\Phi^{(n)}(\vec{w})\)}.
$$

Hence as before, the equation $b_1x_1+\ldots+b_{\n(G)}x_{\n(G)}=0$ has less than
$$
(0.5\n(G))^{\n(G)}(8\n(G))^{4\n(G)(\n(G)-1)^4(m+1)}
$$
solutions with the form
$$
\(Y_1\(\Phi^{(n)}(\vec{w})\),\ldots, Y_{\n(G)}\(\Phi^{(n)}(\vec{w})\)\) \quad (n=0,1,2,\ldots).
$$
Finally, the desired result follows from the one-to-one correspondence between the above vectors and the integers $n\ge 0$.
\end{proof}

\begin{theorem}
\label{thm:Fi1}
Let $\Phi=(F_1,\ldots,F_m)$ with $m\ge 2$ be defined by
$$
F_i=X_i^{s_i}X_{i+1}^{s_{i,i+1}}\cdots X_m^{s_{im}},
$$
with
$$
s_i>1,s_{ij}\ge 0,\quad j=i+1,\ldots,m,
$$
or
$$
s_i\ge 1,s_{ij}\ge 1 \textrm{ for at least one } j=i+1,\ldots,m,
$$
$i=1,\ldots,m-1$,
and
$$
F_m=X_m.
$$
Fix a hypersurface $V=Z(G)$, where
$$
G=\sum_{i_1,\ldots,i_m}a_{i_1,\ldots,i_m}X_1^{i_1}\cdots X_m^{i_m}\in \C[X_1,\ldots,X_m]
$$
with only one monomial of the form $cX_m^{e_m}, c\in \C^*,e_m\ge 1$, such that $G$ has a monomial divisible by $X_j$ for some $1\le j \le m-1$ and zero constant term. 
Then, for any $\w\in (\C^*)^m$ with multiplicatively independent coordinates, we have
$$|\S_{\vec{w}}(\Phi,V)|<(0.5\n(G))^{\n(G)}(8\n(G))^{4\n(G)(\n(G)-1)^4(m+1)}.$$
\end{theorem}
\begin{proof}
As in the proof of Theorem \ref{thm:power di} and under the assumptions,  we can write
$$
G=\sum_{j=1}^{\n(G)} b_jY_j
$$
such that $b_1=c, Y_1=X_m^{e_m}$ and $Y_2$ is not a constant.

For the given point $\w$, noticing $F_m^{(n)}(\w)=w_m$ for any $n\ge 0$, we need to bound the number of integers $n\ge 0$ such that
$$
b_1w_m^{e_m}+\sum_{j=2}^{\n(G)} b_jY_j\(\Phi^{(n)}(\vec{w})\)=0
$$

We first claim that if $n\ne k$, then $Y_2\(\Phi^{(n)}(\vec{w})\)\ne Y_2\(\Phi^{(k)}(\vec{w})\)$. 
Indeed, assume that $n<k$. We note that, by the conditions on $s_i$ and $s_{ij}$ and the fact that the degree $d_{i,n}=\deg F_i^{(n)}$ satisfies
$$
d_{i,n}=s_{i}d_{i,n-1}+s_{i,i+1}d_{i+1,n-1}+\cdots +s_{i,m-1}d_{m-1,n-1}+s_{im},
$$
by induction one easily proves that $\deg F_i^{(n)}<\deg F_i^{(k)}$ for any $1\le i \le m$ (the case $s_i=1$ for all $i=1,\ldots,m-1$ is also proved in~\cite[Lemma 1]{OstShp}).
Thus, by the multiplicative independence of the coordinates of $\w$, we deduce that $Y_2\(\Phi^{(n)}(\vec{w})\)\ne Y_2\(\Phi^{(k)}(\vec{w})\)$.

Hence, similar as before the equation $b_1x_1+\ldots+b_{\n(G)}x_{\n(G)}=0$ has less than
$$
(0.5\n(G))^{\n(G)}(8\n(G))^{4\n(G)(\n(G)-1)^4(m+1)}
$$
solutions with the form
$$
\(w_m^{e_m},Y_2\(\Phi^{(n)}(\vec{w})\),\ldots, Y_{\n(G)}\(\Phi^{(n)}(\vec{w})\)\) \quad (n=0,1,2,\ldots).
$$
Now, the desired result follows as before. 
\end{proof}

\begin{theorem}
\label{thm:Fi2}
Let $\Phi=\(F_1,\ldots,F_m\)$ be defined by
$$
F_i=\prod_{j=1}^{m}X_j^{s_{ij}},\quad  s_{ij}\ge 0
$$
for $i=1,\ldots,m$ and $j=1,\ldots,m$ such that for any $1\le i \le m$ the degree $\deg F_i\ge 2$.
Fix a hypersurface $V=Z(G)$, where
$$
G=\sum_{i_1,\ldots,i_m}a_{i_1,\ldots,i_m}X_1^{i_1}\cdots X_m^{i_m}\in \C[X_1,\ldots,X_m]
$$
with a non-zero constant term $c$.
Then, for any $\w\in (\C^*)^m$ with multiplicatively independent coordinates, we have
$$|\S_{\vec{w}}(\Phi,V)| < (0.5\n(G))^{\n(G)}(8\n(G))^{4\n(G)(\n(G)-1)^4(m+1)}.$$
\end{theorem}

\begin{proof}
As in the proof of Theorem \ref{thm:power di}, we can write
$$
G=\sum_{j=1}^{\n(G)} b_jY_j
$$
with $b_1=c$ and $Y_1=1$. What we need is to
bound the number of integers $n\ge 0$ such that
$$
\sum_{j=1}^{\n(G)} b_jY_j\(\Phi^{(n)}(\vec{w})\)=0.
$$

Since $\deg F_i \ge 2$ for any $1\le i \le m$, we know that
$$
\deg F_i^{(n+1)} > \deg F_i^{(n)}, \quad \textrm{for any $n\ge 0$}.
$$
So, in view of the multiplicative independence of the coordinates of $\vec{w}$, for any integers $n,k$ with $0\le n<k$, we have 
$$
\frac{Y_j\(\Phi^{(k)}(\vec{w})\)}{Y_j\(\Phi^{(n)}(\vec{w})\)}\ne 1, \quad \textrm{for any $2\le j \le \n(G)$}.
$$

Thus as before, the equation $b_1x_1+\ldots+b_{\n(G)}x_{\n(G)}=0$ has less than
$$
(0.5\n(G))^{\n(G)}(8\n(G))^{4\n(G)(\n(G)-1)^4(m+1)}
$$
solutions with the form
$$
\(1,Y_2\(\Phi^{(n)}(\vec{w})\),\ldots, Y_{\n(G)}\(\Phi^{(n)}(\vec{w})\)\) \quad (n=0,1,2,\ldots).
$$
We conclude the proof by noticing the one-to-one correspondence between the above vectors and the integers $n\ge 0$.
\end{proof}

We remark that the assumption on $\Phi$ put in Theorem \ref{thm:Fi2} is reasonable. For example, let $m=2$ and fix such a point $\w$, choose $\Phi=\(X_1,X_2^2\)$ and $G=X_1-w_1$, then for any $n\ge 0$ we have $\Phi^{(n)}(\w)\in Z(G)$.



\section{Comments}
\label{sec:com}

It is quite sure that one can get more partial results concerning the quantitative dynamical Mordell-Lang conjecture by using the methods presented in Section \ref{sec:quan}. It is also likely that several upper bounds in Section \ref{sec:quan} can be improved in some special cases.

However,
the main method based on Corollary~\ref{cor:Sunit} requires that for each $\w \in (\C^*)^m$, all $\Phi^{(n)}(\w)$, $n\ge 1$,  are contained in the same finitely generated group $\Gamma$ of $(\C^*)^m$. Thus, we were able to ensure this property only for monomial systems.

The only non-monomial example for which one can obtain similar results as in Theorems~\ref{thm:power di}--\ref{thm:Fi2} is the following: let $\cF=\cA^{-1}\circ \Phi\circ \cA$, where $\cA$ is a polynomial automorphism and $\Phi$ is any monomial system defined in the results of Section \ref{sec:quan}. Then, for any $n\ge 1$, we have
$$
\cF^{(n)}=\cA^{-1}\circ \Phi^{(n)}\circ \cA.
$$
Thus, for a hypersurface $V=Z(G)$ and a point $\w\in(\C^*)^m$, the point $\cF^{(n)}(\w)\in V$ if and only if $\Phi^{(n)}(\vec{v})\in Z\(G(\cA^{-1})\)$, where $\vec{v}=\cA(\w)$. If the coordinates of $\vec{v}$ are multiplicatively independent, then one can obtain similar results as in Theorems~\ref{thm:power di}--\ref{thm:Fi2}.

It is worth remarking that our methods can also be employed to study the synchronized intersection of two orbits. Indeed,  let $\cF$ and $\cH$ be two polynomial systems from $\C^m$ to itself, then one can ask to bound the size of the subset of integers
$$
\left\{n\ge 0 : \cF^{(n)}(\vec{w}_1)=\cH^{(n)}(\vec{w}_2)\right\},
$$
 where $\vec{w}_1, \vec{w}_2$ are two vectors in $\C^m$. Now, we take $\Phi=(\cF,\cH)$ as in our results (but we see $\cF$ in variables $X_1,\ldots,X_m$ and $\cH$ in variables $Y_1,\ldots,Y_m$, so we have $2m$ variables), and $V=Z(G)$, where $G=\sum_{i=1}^m (X_i-Y_i)$. Clearly, we have
 $$
 \left\{ n\ge 0 : \cF^{(n)}(\vec{w}_1)=\cH^{(n)}(\vec{w}_2) \right\}  \subseteq \S_{(\vec{w}_1, \vec{w}_2)}(\Phi,V),
$$
where $\S_{(\vec{w}_1, \vec{w}_2)}(\Phi,V)$ is defined as in \eqref{eq:intervar}.
So, the problem turns out to bound the size $|S_{(\vec{w}_1, \vec{w}_2)}(\Phi,V)|$, which can be done in many cases by applying the methods in Section \ref{sec:quan}.

In Section \ref{sec:quan}, we consider uniform boundedness of the size of the intersection $\Orb_{\w}(\Phi) \cap V $ for various cases. One can also consider obtaining upper bounds under which there is indeed an integer $n$ such that $\Phi^{(n)}(\w)\in V$, or which are greater than any integer $n$ with $\Phi^{(n)}(\w)\in V$. For example, in Theorem \ref{thm:pd}, if we further assume that $|w_{j_0}| > |w_j|$ and $e_{j_0} \ge e_j$ for any $j \ne j_0$, and define
$$
a=\max_{j\ne j_0} |a_j| \mand \textrm{the index $i$ such that $\left|w_i^{e_i}\right|=\max_{j\ne j_0}\left|w_j^{e_j}\right|$},
$$
then we can easily get a lower bound of $n$ from the following inequality
$$
\left|a_{j_0}w_{j_0}^{e_{j_0}d^n}\right| > (m-1)a\left|w_i^{e_id^n}\right|
$$
such that this bound is larger than any integer $n$ with $\Phi^{(n)}(\w)\in V$.

We also note that if one imposes  in Theorem \ref{thm:power d} that
instead of the  multiplicative independence of the coordinates of
 $\vec{w}\in\C^m$, the absolute values of these  coordinates
 are multiplicatively independent, then there exists $I_0=(i_1,\ldots,i_m)$ such that $X_1^{i_1}\cdots X_m^{i_m}$ is a monomial of $G$ and
$$
\left|\vec{w}^{I_0}\right|> \left|\vec{w}^I\right|$$
with  $I=(j_1,\ldots,j_m)$ for any other monomial  $X_1^{j_1}\ldots X_m^{j_m}$
of $G$ (that is, for  $I \ne I_0$).  This also allows us to
obtain an explicit bound on the largest $n$ such that $\Phi^{(n)}(\w)\in V$.

\section*{Acknowledgements}

The authors would like to thank Igor E. Shparlinski for his valuable suggestions and stimulating discussions, and also for his comments on an early version of the paper. They are also grateful to Joseph H. Silverman for pointing out a serious error in the published version. The research of A.~O. was supported by the
UNSW Vice Chancellor's Fellowship and of M.~S. by the Australian Research Council Grant DP130100237.


\begin{thebibliography}{99}

\bibitem{AV1} F. Amoroso and E. Viada, `Small points on subvarieties of a torus', {\it Duke
Math. J.}, {\bf 150} (2009), 407--442.

\bibitem{AV2}
F. Amoroso and E. Viada, `On the zeros of linear recurrence sequences',
\textit{Acta Arith.}, \textbf{147} (2011), 387--396.

\bibitem{AnKhr} V. Anashin and A. Khrennikov,
`Applied algebraic dynamics', Walter de Gruyter, Berlin. 2009.

\bibitem{Bell}
J.P. Bell, `A generalised Skolem-Mahler-Lech theorem for affine varieties', \textit{J. Lond. Math. Soc.}, \textbf{73} (2006), 367--379.

\bibitem{BGT}
J.P. Bell, D. Ghioca and T.J. Tucker, `The dynamical Mordell-Lang problem for \'etale maps', \textit{Amer. J. Math.}, \textbf{132} (2010),  1655--1675.

\bibitem{BGT14}
J.P. Bell, D. Ghioca and T.J. Tucker, `The dynamical Mordell-Lang problem for Noetherian spaces', preprint, 2014, \url{http://www.math.ubc.ca/~dghioca/papers/dml-v20.pdf}.

\bibitem{BGKT}
R. Benedetto, D. Ghioca, P. Kurlberg and T. Tucker, `A case of the dynamical Mordell-Lang conjecture',
\textit{Math. Ann.}, \textbf{352} (2012), 1--26.

\bibitem{BT}
D. Berend and T. Tassa, `Improved bounds on Bell numbers and on moments of sums of random variables', \textit{Probability and Mathematical Statistics}, \textbf{30} (2010), 185--205.

\bibitem{BK}
E. Bombieri and N.M. Katz, `A note on lower bounds for Frobenius traces', \textit{Enseign. Math.}, \textbf{56} (2010), 203--227.

\bibitem{DOSS} C. D'Andrea, A. Ostafe, M. Sombra and I. Shparlinski, `Modular reduction of systems of polynomial equations and
algebraic dynamical systems', preprint, 2014.

\bibitem{Denis}
L. Denis, `G\'eom\'etrie et suites r\'ecurrentes', \textit{Bull. Soc. Math. France}, \textbf{122} (1994), 13--27.

\bibitem{Dubickas2012}
A.~Dubickas, \emph{Roots of unity as quotients of two roots of a polynomial}, J. Aust. Math. Soc., {\bf 92} (2012), 137--144.

\bibitem{Ever}
J.-H. Evertse, `The number of solutions of decomposable form equations', \textit{Invent. Math.}, \textbf{122} (1995), 559--601.

\bibitem{Fak}
N. Fakhruddin, `The algebraic dynamics of generic endomorphisms of $\PP^n$', \textit{Algebra Number Theory}, \textbf{8} (2014), 587--608.

\bibitem{GT2008}
D. Ghioca and T.J. Tucker, `A dynamical version of the Mordell-Lang conjecture for the additive group',
\textit{Compos. Math.}, \textbf{144} (2008), 304--316.

\bibitem{GT}
D. Ghioca and T.J. Tucker, `Periodic points, linearizing maps, and the dynamical Mordell-Lang problem',
\textit{J. Number Theory},  \textbf{129} (2009), 1392--1403.

\bibitem{GTZ1}
D. Ghioca, T. Tucker and M. Zieve, `Intersections of polynomial orbits, and a dynamical Mordell-Lang conjecture', {\it Invent. Math.\/}, {\bf 171} (2008), 463--483.

 \bibitem{OstShp} A. Ostafe and I. E. Shparlinski, `On the degree growth
 in some polynomial dynamical  systems and nonlinear pseudorandom number generators',
 {\it Math. Comp.}, {\bf 79} (2010),  501--511.

\bibitem{PSc} A.J. van der Poorten and H.P. Schlickewei, `Zeros of recurrence sequences', {\it Bull. Aust. Math. Soc.}, {\bf 44} (1991), 215--223.

\bibitem{SchSch} H. P. Schlickewei and W. M. Schmidt, `The number of solutions of polynomial-exponential equations', {\it Compositio Math.}, {\bf 120} (2000), 193--225.

\bibitem{Schm} K.~Schmidt,
{\it Dynamical systems of algebraic origin\/},
Progress in Math., v.128, Birkh{\"a}user Verlag, Basel, 1995.

\bibitem{Silv1}
J.H. Silverman, {\it The arithmetic of dynamical systems\/},
Springer, New York, 2007.

\bibitem{SiVi}
J.H. Silverman and B. Viray, `On a uniform bound for the number of exceptional linear subvarieties in the dynamical Mordell-Lang conjecture',
{\it Math. Res. Letters\/},  {\bf  20} (2013),  547--566.

\bibitem{Xie}
 J. Xie, `Dynamical Mordell-Lang conjecture for birational
polynomial morphisms on $\A^2$', \textit{Math. Ann.}, \textbf{360} (2014), 457--480.

\end{thebibliography}
\end{document}